\documentclass[a4paper,reqno,11pt,oneside]{amsart}

\usepackage{latexsym}
\usepackage{float}
\usepackage{amsfonts,amssymb,latexsym,xspace,epsfig,graphics,color}
\usepackage{amsmath,enumerate,stmaryrd,xy,stackrel}
\usepackage{pgfplots}
\usepackage[footnotesize]{caption}
\usepackage{indentfirst}
\usepackage[T1]{fontenc}
\usepackage{a4wide}
\usepackage{url}
\usepackage{tikz}
\usepackage{color}
\usepackage{graphpap}
\usepackage{epsfig}
\usepackage{psfrag}
\usepackage{graphicx}
\usepackage{subfigure}
\usetikzlibrary{arrows,snakes,backgrounds,trees,shapes}
\usepackage{multirow}
\usepackage{threeparttable}
\usepackage{float}
\usepackage{rotating}
\usepackage{cite}

\theoremstyle{plain}
\newtheorem{theorem}{Theorem}[section]
\newtheorem{corollary}{Corollary}[section]
\newtheorem{lemma}{Lemma}[section]
\newtheorem{prop}{Proposition}[section]

\theoremstyle{definition}
\newtheorem{definition}{Definition}[section]
\newtheorem{exa}{Example}[section]

\theoremstyle{remark}
\newtheorem{remark}{Remark}

\numberwithin{equation}{section}
\numberwithin{figure}{section}

\begin{document}

\title[Stochastic rumors on random trees]{Stochastic rumors on random trees}

\date{}

\author[Valdivino V. Junior]{Valdivino V. Junior}  
\address{Valdivino V. Junior. Universidade Federal de Goias, Campus Samambaia, CEP 74001-970, Goi\^ania, GO, Brazil. E-mail: vvjunior@ufg.br}

\author[Pablo Rodriguez]{Pablo M. Rodriguez} 
\address{Pablo Rodriguez. Universidade Federal de Pernambuco, Av. Prof. Moraes Rego, 1235. Cidade Universit\'aria, CEP 50670-901, Recife, PE, Brazil. E-mail: pablo@de.ufpe.br}

\author[Adalto Speroto]{Adalto Speroto} 
\address{Adalto Speroto. Universidade de S\~ao Paulo, Caixa Postal 668, CEP 13560-970, S\~ao Carlos, SP, Brazil. E-mail: speroto@usp.br}



\begin{abstract}
The Maki-Thompson rumor model is defined by assuming that a population represented by a graph is subdivided into three classes of individuals; namely, ignorants, spreaders and stiflers. A spreader tells the rumor to any of its nearest ignorant neighbors at rate one. At the same rate, a spreader becomes a stifler after a contact with other nearest neighbor spreaders, or stiflers. In this work we study the model on random trees. As usual we define a critical parameter of the model as the critical value around which the rumor either becomes extinct almost-surely or survives with positive probability. We analyze the existence of phase-transition regarding the survival of the rumor, and we obtain estimates for the mean range of the rumor. The applicability of our results is illustrated with examples on random trees generated from some well-known discrete distributions. 

\end{abstract}

\maketitle

\section{Introduction}

\subsection{The first models for rumor spreading}\label{ss:intro}

In this work we contribute with the literature of mathematical models for information transmission. As far as we know many of the well-known basic models appeared in the earlier 1960. One of them is widely cited to nowadays and was communicated by Daley and Kendall in \cite{daley_nature}. Such a model appeared as an alternative to the well-known SIR epidemic model for describing the propagation of infectious diseases. While a SIR model assumes that a population is subdivided into susceptible, infected and removed individuals, in Daley and Kendall's model, these classes are called ignorants (those not aware of the rumor), spreaders (who are spreading it) or stiflers (who know the rumor but have ceased communicating it after meeting somebody who has already heard it), respectively. Although the mechanisms for transmission are similar; indeed, ignorants in rumors are similar to susceptible in epidemics, spreaders are similar to infected cases, and stiflers correspond to removals; the difference between these two processes is in the stopping mechanisms. In the propagation of a disease an infected individual becomes removed only after a random time, which is in general independent of what happens with other individuals. On the other hand, in the transmission of a rumor a spreader stops talking about the rumor right after getting involved in a contact with another individual who already knows it. In other words, in a rumor process, a spreader becomes a stifler at a rate which depends on the number of non-ignorant individuals in the population. The communication in \cite{daley_nature} and the results obtained in \cite{kendall} gave rise to a theory of mathematical models for rumor transmission. For a review of some works dealing with rumors we refer the reader to \cite[Section 1]{AECP} and the references therein.

Here we deal with the simplified version of the Daley--Kendall model due to Maki and Thompson in \cite{MT}. The Maki--Thompson model describes the spreading of a rumor on a closed homogeneously mixed population, subdivided into the three classes of individuals mentioned before: ignorants, spreaders, and stiflers. Formally, the Maki--Thompson model may be written as the continuous-time Markov chain $\{(X^{N}(t), Y^{N}(t))\}_{t \in [0,\infty)}$ which evolves according to the following transitions and rates
\begin{equation}\label{eq:transMT}
\begin{array}{cc}
\text{transition} \quad &\text{rate} \\[0.1cm]
(-1, 1) \quad &X Y, \\[0.1cm]
(0, -1) \quad &Y (N -1 - X).
\end{array}
\end{equation}

For $t\geq 0$, the random variables $X^{N}(t)$, $Y^{N}(t)$ and $Z^{N}(t)$ denote, respectively, the number of ignorants, spreaders and stiflers at time $t$. The basic version of the model assumes that $X^{N}(0) = N-1$, $Y^{N}(0) = 1$, $Z^{N}(0) = 0$, and $X^{N}(t) + Y^{N}(t) + Z^{N}(t) = N $ for all~$t$. Thus defined, \eqref{eq:transMT} means that if the process is in state $(i,j)$ at time $t$, then the probabilities that it jumps to states $(i-1,j+1)$ or $(i,j-1)$ at time $t+h$ are, respectively, $i\,j\,h + o(h)$ and $j (N - 1 - i)\,h + o(h)$, where $o(h)$ represents a function such that $\lim_{h\to 0}o(h)/h =0$. This describes the situation in which individuals interact by contacts initiated by the spreaders: the two possible transitions in \eqref{eq:transMT} correspond to spreader-ignorant, and spreader-(spreader or  stifler) interactions. In the first case, the spreader tells the rumor to the ignorant, who becomes a spreader. The other transition represents the transformation of a spreader into a stifler after initiating a meeting with a non-ignorant. Note that the last event describes the loss of interest in propagating the rumor derived from learning that it is already known by the other individual in the meeting. This is the main difference between rumor models and mathematical models for infectious diseases. 

A question of interest when studying a stochastic rumor model on a finite population is what we can say about the remaining proportion of people who never hear the rumor. Note that, since the process eventually ends, this is one way of having a measure of the ``size'' of the rumor. The first rigorous results in this direction are limit theorems for the remaining proportion of ignorants when the process ends, as the population size grows to $\infty$. It has been proved that, for both the Daley--Kendall and the Maki--Thompson models, such final  proportion of ignorants equals approximately $20\%$, see \cite{kendall,sudbury,watson}. In recent works, the main approach to deal with such a question for generalized rumor models is a suitable application of the theory of convergence of density dependent Markov chains. This allows to find a system of ordinary differential equations whose solution represents an approximation, for sufficiently large $N$, of a scaled version of the entire trajectories of the stochastic model. Moreover, these solutions may be seen as good approximations for the proportion of ignorants and spreaders, respectively, at time $t$, for $t>0$ and $N$ large enough. An extensive treatment of rumor models can be found in \cite[Chapter~5]{dg}. In addition, for a deeper discussion of the use of density dependent Markov chains for studying rumor models we refer the reader to \cite{arrudaroyal,lebensztayn/machado/rodriguez/2011a,lebensztayn/machado/rodriguez/2011b,AECP}.

\subsection{Changing the structure of the population. Why random trees?}

The Maki--Thompson model studied by the references mentioned in Subsection \ref{ss:intro} is formulated assuming a closed homogeneously mixed population. If we represent the population with a graph, that assumption means that the considered graph is the complete graph, see Figure \ref{fig:graphs}(a). With such an interpretation, vertices represent individuals and edges represent possible contacts between them. By considering a complete graph, the theoretical analysis of the model is simplified, in some sense, because it is enough to deal with the proportions of individuals in each one of the classes. This is why it is less difficult to deal with generalizations of the model in homogeneously mixed populations. However, when one want to consider a different structure of the population the approach should be changed and, in most cases, it depends on the considered graph.

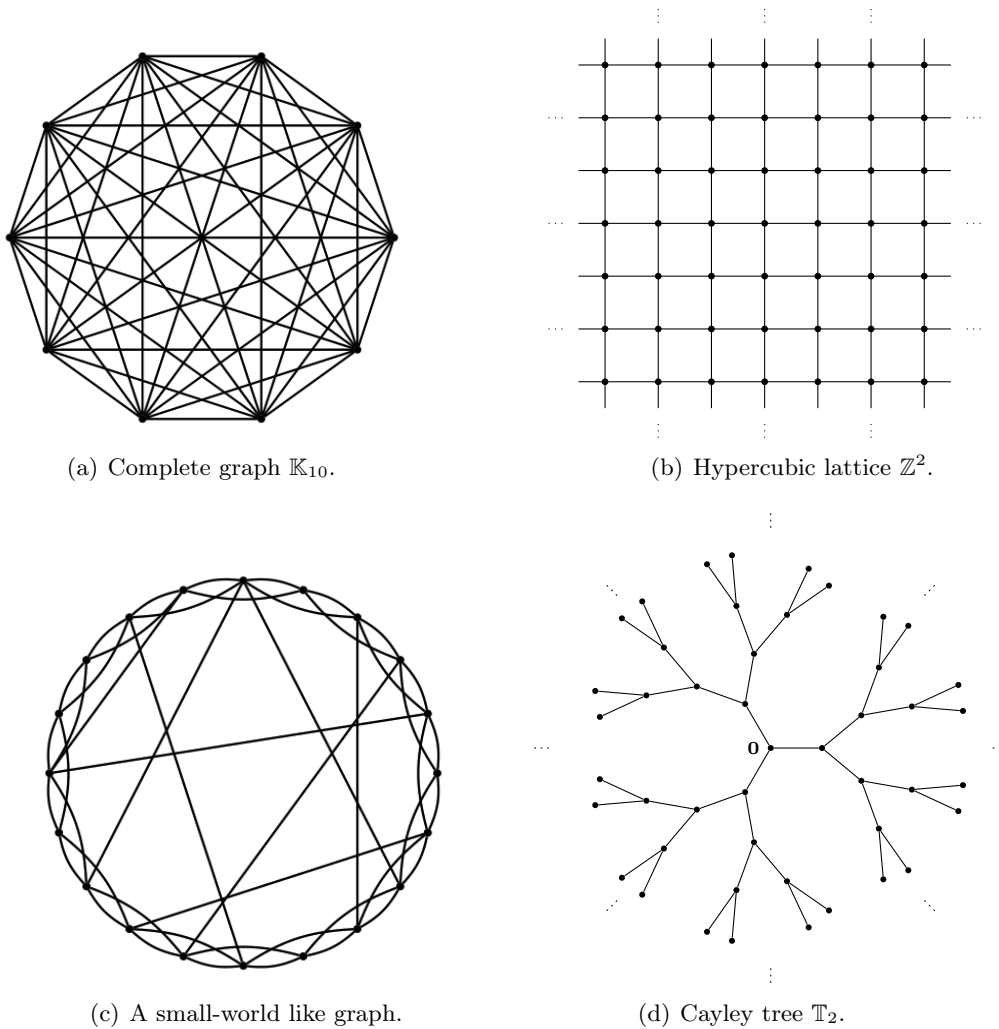
\begin{figure}[h!]
\begin{center}
\subfigure[][Complete graph $\mathbb{K}_{10}$.]{

\includegraphics[width=5.5cm]{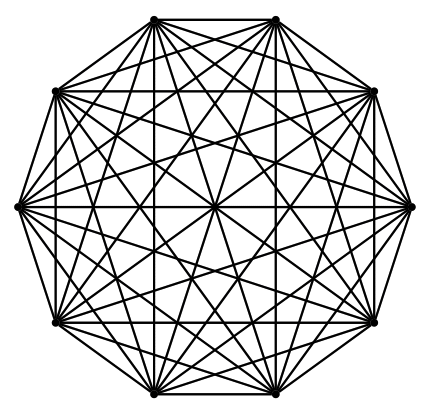}

}\qquad\qquad\subfigure[][Hypercubic lattice $\mathbb{Z}^2$.]{
\begin{tikzpicture}[scale=0.7]
\draw (-3.5,3) -- (3.5,3);
\draw (-3.5,2) -- (3.5,2);
\draw (-3.5,1) -- (3.5,1);
\draw (-3.5,0) -- (3.5,0);
\draw (-3.5,-1) -- (3.5,-1);
\draw (-3.5,-2) -- (3.5,-2);
\draw (-3.5,-3) -- (3.5,-3);
\draw[dotted] (-3.8,0) -- (-4.1,0);
\draw[dotted] (-3.8,2) -- (-4.1,2);
\draw[dotted] (-3.8,-2) -- (-4.1,-2);
\draw[dotted] (3.8,0) -- (4.1,0);
\draw[dotted] (3.8,2) -- (4.1,2);
\draw[dotted] (3.8,-2) -- (4.1,-2);
\draw[dotted] (0,-3.8) -- (0,-4.1);
\draw[dotted] (2,-3.8) -- (2,-4.1);
\draw[dotted] (-2,-3.8) -- (-2,-4.1);
\draw[dotted] (0,3.8) -- (0,4.1);
\draw[dotted] (2,3.8) -- (2,4.1);
\draw[dotted] (-2,3.8) -- (-2,4.1);

\draw (-3,-3.5) -- (-3,3.5);
\draw (-2,-3.5) -- (-2,3.5);
\draw (-1,-3.5) -- (-1,3.5);
\draw (0,-3.5) -- (0,3.5);
\draw (1,-3.5) -- (1,3.5);
\draw (2,-3.5) -- (2,3.5);
\draw (3,-3.5) -- (3,3.5);


\filldraw (-3,3) circle (1.5pt);
\filldraw (-3,2) circle (1.5pt);
\filldraw (-3,1) circle (1.5pt);
\filldraw (-3,0) circle (1.5pt);
\filldraw (-3,-1) circle (1.5pt);
\filldraw (-3,-2) circle (1.5pt);
\filldraw (-3,-3) circle (1.5pt);
\filldraw (-2,3) circle (1.5pt);
\filldraw (-2,2) circle (1.5pt);
\filldraw (-2,1) circle (1.5pt);
\filldraw (-2,0) circle (1.5pt);
\filldraw (-2,-1) circle (1.5pt);
\filldraw (-2,-2) circle (1.5pt);
\filldraw (-2,-3) circle (1.5pt);
\filldraw (3,3) circle (1.5pt);
\filldraw (3,2) circle (1.5pt);
\filldraw (3,1) circle (1.5pt);
\filldraw (3,0) circle (1.5pt);
\filldraw (3,-1) circle (1.5pt);
\filldraw (3,-2) circle (1.5pt);
\filldraw (3,-3) circle (1.5pt);
\filldraw (2,3) circle (1.5pt);
\filldraw (2,2) circle (1.5pt);
\filldraw (2,1) circle (1.5pt);
\filldraw (2,0) circle (1.5pt);
\filldraw (2,-1) circle (1.5pt);
\filldraw (2,-2) circle (1.5pt);
\filldraw (2,-3) circle (1.5pt);
\filldraw (-1,3) circle (1.5pt);
\filldraw (-1,2) circle (1.5pt);
\filldraw (-1,1) circle (1.5pt);
\filldraw (-1,0) circle (1.5pt);
\filldraw (-1,-1) circle (1.5pt);
\filldraw (-1,-2) circle (1.5pt);
\filldraw (-1,-3) circle (1.5pt);
\filldraw (1,3) circle (1.5pt);
\filldraw (1,2) circle (1.5pt);
\filldraw (1,1) circle (1.5pt);
\filldraw (1,0) circle (1.5pt);
\filldraw (1,-1) circle (1.5pt);
\filldraw (1,-2) circle (1.5pt);
\filldraw (1,-3) circle (1.5pt);
\filldraw (0,3) circle (1.5pt);
\filldraw (0,2) circle (1.5pt);
\filldraw (0,1) circle (1.5pt);
\filldraw (0,0) circle (1.5pt);
\filldraw (0,-1) circle (1.5pt);
\filldraw (0,-2) circle (1.5pt);
\filldraw (0,-3) circle (1.5pt);

\draw [->,line width=0.4mm,white] (3.1,1.5) to [out=10, in=180]  (5.3,2);

\end{tikzpicture}}

\subfigure[][A small-world like graph.]{

\includegraphics[width=5.7cm]{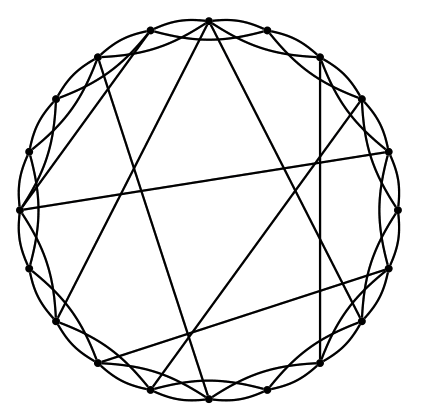}

}\subfigure[][Cayley tree $\mathbb{T}_2$.]{

\tikzstyle{level 1}=[sibling angle=120]
\tikzstyle{level 2}=[sibling angle=80]
\tikzstyle{level 3}=[sibling angle=60]
\tikzstyle{level 4}=[sibling angle=30]

\tikzstyle{edge from parent}=[segment length=0.6cm,segment angle=10,draw]
\begin{tikzpicture}[scale=0.45,grow cyclic,shape=circle,minimum size = 1pt,inner sep=0.8pt,level distance=15mm,
                    cap=round]
\node[fill] {} child [\A] foreach \A in {black,black,black}
    { node[fill] {} child [color=\A!50!\B] foreach \B in {black,black}
        { node[fill] {} child [color=\A!50!\B!50!\C] foreach \C in {black,black}
            { node[fill] {} child [color=\A!50!\B!50!\C!50!\D] foreach \D in {black,black}
            { node[fill] {} }
        }
    }
    };
    
\node at (-0.5,0) {\tiny ${\bf 0}$};

\draw[dotted] (6.5,0) -- (7,0);
\draw[dotted] (-6.5,0) -- (-7,0);
\draw[dotted] (0,6.5) -- (0,7);
\draw[dotted] (0,-6.5) -- (0,-7);
\draw[dotted] (4.5,4.5) -- (4.8,4.8);
\draw[dotted] (4.5,-4.5) -- (4.8,-4.8);
\draw[dotted] (-4.5,-4.5) -- (-4.8,-4.8);
\draw[dotted] (-4.5,4.5) -- (-4.8,4.8);
\end{tikzpicture}}
\caption{Some families of graphs used to represent a population in rigorous mathematical models for rumor spreading.}\label{fig:graphs}
\end{center}
\end{figure}

In a general graph the model may be defined, roughly speaking, by assuming that the set of vertices representing the population is subdivided into three classes; namely, ignorants, spreaders and stiflers. Then a spreader tells the rumor to any of its nearest ignorant neighbors at rate one, and, at the same (but could be other) rate, a spreader becomes a stifler after a contact with other nearest neighbor spreaders, or stiflers. The mathematical model, which is stochastic, is an interacting particle system. We shall formalize this type of model in Section \ref{s:model}. The first rigorous result for a modified version of the Maki--Thompson in a non-complete graph is due to \cite{CRS}, where the authors represent the population with an hypercubic lattice $\mathbb{Z}^d$, for $d\geq 2$, see Figure \ref{fig:graphs}(b), and obtain results of phase-transition according to the values of general rates of interactions between the classes of individuals. In such a work the approach relies on coupling techniques and a combination of results coming from oriented percolation and the so-called contact process. See \cite{grimmett} for a review of these subjects. Note that $\mathbb{Z}^d$ is an infinite graph. On the other hand, \cite{EAP} studies the Maki--Thompson model on a small-world graph. The graph is constructed by starting from a $k$-regular ring and by inserting, in the average, $c$ additional links in such a way that $k$ and $c$ are tuneable parameters for the population structure; see Figure \ref{fig:graphs}(c). The authors prove that this system exhibits a kind of transition between regimes of localization (where the final number of stiflers is at most logarithmic in the population size) and propagation (where the final number of stiflers grows algebraically with the population size) depending of the network parameter $c$ being small our large enough, respectively. The approach for that analysis was, again, coupling techniques, but now combined together with the theory of branching processes. More recently, in order to understand the behavior of the model on trees, \cite{junior} consider the generalization called the Maki--Thomposn model with $k$-stifling on infinite Cayley trees $\mathbb{T}_d$, for $d\geq 2$; see Figure \ref{fig:graphs}(d). These are deterministic infinite trees such that any vertex has degree $d+1$. The theory of branching processes was used for this study as well and the main results are related to the analysis of the existence of a critical $d$ around which the rumor either becomes extinct almost-surely or survives with positive probability. In addition, the authors obtain information about the range of spreading. 
In this work we extend the approach of \cite{junior} to deal with stochastic rumors on random trees, which can be seen as the family trees produced by a branching process. In words, we start with one vertex which is the root of the tree and let it connect according to a given discrete distribution. Each of these neighbors (if there are any) then connect with new vertices independently with the same law, and so on forever or until some generation goes extinct. These trees are called Galton-Watson trees, or random trees for short. 
This case gains in interest if we realize that random trees naturally appears ``embebed'' in many models of random graphs and complex networks. It is well-known that many of such models, which are more appropriated to represent a population, behave locally as trees. Up to now, stochastic rumors in random graphs or complex networks were studied mainly by mean of computational simulations or mean-field approximations. See for example \cite{arruda3,MNP-PRE2004,moreno-PhysA2007,zanette02}. As a sideline, we point out that recent research suggest that random trees are suitable structures to represent the spreading of information between population, see for example \cite{gleeson} where branching processes are used to model propagation of information on Twitter.  On the other hand, for a review of some long-range propagation models on (random) trees we refer the reader to \cite{cone,cone2}.    

We analyze the existence of a critical value, depending of the law generating the tree, around which the rumor either becomes extinct almost-surely or survives with positive probability. Following \cite{junior}, we also obtain information about the range of spreading. Our work is complemented with some examples and a discussion for further research.  

The paper is organized as follows. In Section 2 we formalize the mathematical model and present the main results. Subsection 2.1 is dedicated to a review of definitions for the model on infinite Cayley trees, and from Subsection 2.2 onward the model on random trees and all the results of our work are presented. Section 3 es devoted to the proof of our main theorems and Section 4 is left to a discussion of further research.

\section{The model and main results}\label{s:model}

\subsection{The model on an infinite Cayley tree}

In this work we shall work with infinite trees $\mathbb{T}=(\mathcal{V},\mathcal{E})$. As usual $ \mathcal{V} $ stands for the set of vertices and $\mathcal{E} \subset \{\{u,v\}: u,v \in \mathcal{V}, u \neq v\}$ stands for the set of edges. We shall abuse notation by writing $\mathcal{V}=\mathbb{T}$. We consider rooted trees in the sense that we shall identify one vertex as the root and denote it by ${\bf 0}$, see Figure \ref{fig:graphs}(d). If $\{u,v\}\in \mathcal{E}$, we say that $u$ and $v$ are neighbors, and we denote it by $u\sim v$. The degree of a vertex $v$, denoted by $deg(v)$, is the number of its neighbors. A path in $\mathbb{T}$ is a finite sequence $v_0, v_1, \dots, v_n $ of distinct vertices such that $ v_i \sim v_{i+1} $ for each $i$, and a ray in $\mathbb{T}$ is a path with infinite vertices starting at ${\bf 0}$. For any tree, there is a unique path connecting any pair of distinct vertices $u$ and $v$. Therefore we define the distance between them, denoted by $d(u,v)$, as the number of edges in such path. We denote by $\mathbb{T}_d$ the infinite Cayley tree of coordination number $d+1$, with $d\geq 2$. The notation is because the same graph is known as $(d+1)$-dimensional homogeneous tree, which is a graph with an infinite number of vertices, without cycles and such that every vertex has degree $d+1$. For each $v\in \mathcal{V}$ define $|v|:=d({\bf{0}},v)$. For $ u,v \in \mathcal{V} $, we say that $u\leq v$ if $u$ is one of the vertices of the path connecting ${\bf{0}}$ and $v$; $u<v$ if $u\leq v$ and $u\neq v$. We call $v$ a \textit{successor} of $u$ if $u\leq v$ and $u \sim v$. We denote by $\partial \mathbb{T}_{n}$ the set of vertices at distance $n$ from the root. That is, $\partial \mathbb{T}_{n}= \{v \in \mathbb{T}: |v|=n\}$. 


The Maki-Thompson model on an infinite tree $\mathbb{T}$ may be defined as a continuous-time Markov process $(\eta_t)_{t\geq 0}$ with states space $\mathcal{S}=\{0,1,2\}^{\mathbb{T}}$; i.e. at time $t$ the state of the process is some function $\eta_t: \mathbb{T} \longrightarrow \{0,1,2\}$. We assume that each vertex $x \in \mathbb{T}$ represents an individual, which is said to be an ignorant if $\eta(x)=0,$ a spreader if $\eta(x)=1$, and a stifler if $\eta(x)=2.$ Then, if the system is in configuration $\eta \in \mathcal{S},$ the state of vertex $x$ changes according to the following transition rates

\begin{equation}\label{rates}
\begin{array}{rclc}
&\text{transition} &&\text{rate} \\[0.1cm]
0 & \rightarrow & 1, & \hspace{.5cm}  n_{1}(x,\eta),\\[.2cm]
1 & \rightarrow & 2, & \hspace{.5cm}   n_{1}(x,\eta) + n_{2}(x,\eta),
\end{array}
\end{equation}

\noindent where $$n_i(x,\eta)= \sum_{y\sim x} 1\{\eta(y)=i\}$$ 
is the number of nearest neighbors of vertex $x$ in state $i$ for the configuration $\eta$, for $i\in\{1,2\}.$ Formally, \eqref{rates} means that if the vertex $x$ is in state, say, $0$ at time $t$ then the probability that it will be in state $1$ at time $t+h$, for $h$ small, is $n_{1}(x,\eta) h + o(h)$, where $o(h)$ represents a function such that $\lim_{h\to 0} o(h)/h = 0$. Note that the rates in \eqref{rates} represent how the changes of states of individuals depend on the states of its neighbors. While the change of state of an ignorant is influenced by its spreader neighbors, the change of state for a spreader is influenced by the number of non-ignorant neighbors. We point out that stiflers do not interact with ignorants. See Figure \ref{fig:realization} for an illustration of these transitions.

\begin{figure}[h!]
    \centering
\begin{tikzpicture}

\filldraw [black] (0,4) circle (3pt);
\node at (0.9,4) {\footnotesize : ignorant};

\filldraw [red!80!black] (3,4) circle (3pt);
\node at (3.9,4) {\footnotesize : spreader};

\filldraw [blue!30!gray] (6,4) circle (3pt);
\node at (6.7,4) {\footnotesize : stifler};
\draw[gray] (-1,3) -- (9,3);

\node at (-2,-1.5) {\footnotesize $\bf (a)$};

\node[rotate=330] at (0,0) {

\begin{tikzpicture}

\draw[thick] (0,0) -- (2,0) -- (4,0);
\draw[thick] (2,0) -- (4,1.5);
\draw[thick] (2,0) -- (4,-1.5);
\draw[thick,gray,dashed] (4,0) -- (4.5,0);
\draw[thick,gray,dashed] (4,0) -- (4.5,0.5);
\draw[thick,gray,dashed] (4,0) -- (4.5,-0.5);
\draw[thick,gray,dashed] (4,1.5) -- (4.5,1.95);
\draw[thick,gray,dashed] (4,1.5) -- (4,2);
\draw[thick,gray,dashed] (4,1.5) -- (4.5,1.5);
\draw[thick,gray,dashed] (4,-1.5) -- (4.5,-1.95);
\draw[thick,gray,dashed] (4,-1.5) -- (4,-2);
\draw[thick,gray,dashed] (4,-1.5) -- (4.5,-1.5);
\draw[thick,gray,dashed] (0,0) -- (-0.5,0);
\draw[thick,gray,dashed] (0,0) -- (-0.5,0.5);
\draw[thick,gray,dashed] (0,0) -- (-0.5,-0.5);

\filldraw [red!80!black] (0,0) circle (3pt);

\draw [->,line width=0.4mm,gray] (0.1,0.2) to [out=30, in=150]  (1.9,0.2);

\filldraw [black] (2,0) circle (3pt);
\filldraw [black] (4,0) circle (3pt);
\filldraw [black] (4,1.5) circle (3pt);
\filldraw [black] (4,-1.5) circle (3pt);
\end{tikzpicture}};

\node at (6,-1.5) {\footnotesize  $\bf (b)$};

\node[rotate=330] at (8,0) {

\begin{tikzpicture}

\draw[thick] (0,0) -- (2,0) -- (4,0);
\draw[thick] (2,0) -- (4,1.5);
\draw[thick] (2,0) -- (4,-1.5);
\draw[thick,gray,dashed] (4,0) -- (4.5,0);
\draw[thick,gray,dashed] (4,0) -- (4.5,0.5);
\draw[thick,gray,dashed] (4,0) -- (4.5,-0.5);
\draw[thick,gray,dashed] (4,1.5) -- (4.5,1.95);
\draw[thick,gray,dashed] (4,1.5) -- (4,2);
\draw[thick,gray,dashed] (4,1.5) -- (4.5,1.5);
\draw[thick,gray,dashed] (4,-1.5) -- (4.5,-1.95);
\draw[thick,gray,dashed] (4,-1.5) -- (4,-2);
\draw[thick,gray,dashed] (4,-1.5) -- (4.5,-1.5);
\draw[thick,gray,dashed] (0,0) -- (-0.5,0);
\draw[thick,gray,dashed] (0,0) -- (-0.5,0.5);
\draw[thick,gray,dashed] (0,0) -- (-0.5,-0.5);

\filldraw [red!80!black] (0,0) circle (3pt);
\filldraw [red!80!black] (2,0) circle (3pt);

\draw [->,line width=0.4mm,gray] (2,0.2) to [out=90, in=180]  (3.8,1.5);

\filldraw [black] (4,0) circle (3pt);
\filldraw [black] (4,1.5) circle (3pt);
\filldraw [black] (4,-1.5) circle (3pt);
\end{tikzpicture}};

\node at (-2,-7.5) {\footnotesize   $\bf (c)$};

\node[rotate=330] at (0,-6) {

\begin{tikzpicture}

\draw[thick] (0,0) -- (2,0) -- (4,0);
\draw[thick] (2,0) -- (4,1.5);
\draw[thick] (2,0) -- (4,-1.5);
\draw[thick,gray,dashed] (4,0) -- (4.5,0);
\draw[thick,gray,dashed] (4,0) -- (4.5,0.5);
\draw[thick,gray,dashed] (4,0) -- (4.5,-0.5);
\draw[thick,gray,dashed] (4,1.5) -- (4.5,1.95);
\draw[thick,gray,dashed] (4,1.5) -- (4,2);
\draw[thick,gray,dashed] (4,1.5) -- (4.5,1.5);
\draw[thick,gray,dashed] (4,-1.5) -- (4.5,-1.95);
\draw[thick,gray,dashed] (4,-1.5) -- (4,-2);
\draw[thick,gray,dashed] (4,-1.5) -- (4.5,-1.5);
\draw[thick,gray,dashed] (0,0) -- (-0.5,0);
\draw[thick,gray,dashed] (0,0) -- (-0.5,0.5);
\draw[thick,gray,dashed] (0,0) -- (-0.5,-0.5);

\filldraw [red!80!black] (0,0) circle (3pt);

\draw [->,line width=0.4mm,gray] (0.1,0.2) to [out=30, in=150]  (1.9,0.2);

\filldraw [red!80!black] (2,0) circle (3pt);
\filldraw [black] (4,0) circle (3pt);
\filldraw [red!80!black] (4,1.5) circle (3pt);
\filldraw [black] (4,-1.5) circle (3pt);
\end{tikzpicture}};

\node at (6,-7.5) {\footnotesize $\bf (d)$};

\node[rotate=330] at (8,-6) {

\begin{tikzpicture}

\draw[thick] (0,0) -- (2,0) -- (4,0);
\draw[thick] (2,0) -- (4,1.5);
\draw[thick] (2,0) -- (4,-1.5);
\draw[thick,gray,dashed] (4,0) -- (4.5,0);
\draw[thick,gray,dashed] (4,0) -- (4.5,0.5);
\draw[thick,gray,dashed] (4,0) -- (4.5,-0.5);
\draw[thick,gray,dashed] (4,1.5) -- (4.5,1.95);
\draw[thick,gray,dashed] (4,1.5) -- (4,2);
\draw[thick,gray,dashed] (4,1.5) -- (4.5,1.5);
\draw[thick,gray,dashed] (4,-1.5) -- (4.5,-1.95);
\draw[thick,gray,dashed] (4,-1.5) -- (4,-2);
\draw[thick,gray,dashed] (4,-1.5) -- (4.5,-1.5);
\draw[thick,gray,dashed] (0,0) -- (-0.5,0);
\draw[thick,gray,dashed] (0,0) -- (-0.5,0.5);
\draw[thick,gray,dashed] (0,0) -- (-0.5,-0.5);

\filldraw [blue!30!gray] (0,0) circle (3pt);

\draw [->,line width=0.4mm,white] (0.1,0.2) to [out=30, in=150]  (1.9,0.2);

\filldraw [red!80!black] (2,0) circle (3pt);
\filldraw [black] (4,0) circle (3pt);
\filldraw [red!80!black] (4,1.5) circle (3pt);
\filldraw [black] (4,-1.5) circle (3pt);
\end{tikzpicture}};

\end{tikzpicture}

    \caption{Possible realization of the MT-model on a tree $\mathbb{T}$. The vertices of the tree represent individuals belonging to one of the following tree classes: ignorants, spreaders and stiflers, which we identify by black, red, and blue vertices, respectively. (a) A spreader tells the rumor to any of its nearest ignorant neighbors at rate one. (b) After the previous interaction the contacted ignorant becomes a spreader and starts to transmit the information. (c)-(d) At the same rate, a spreader becomes a stifler after a contact with other nearest neighbor spreaders, or stiflers.}
    \label{fig:realization}
\end{figure}
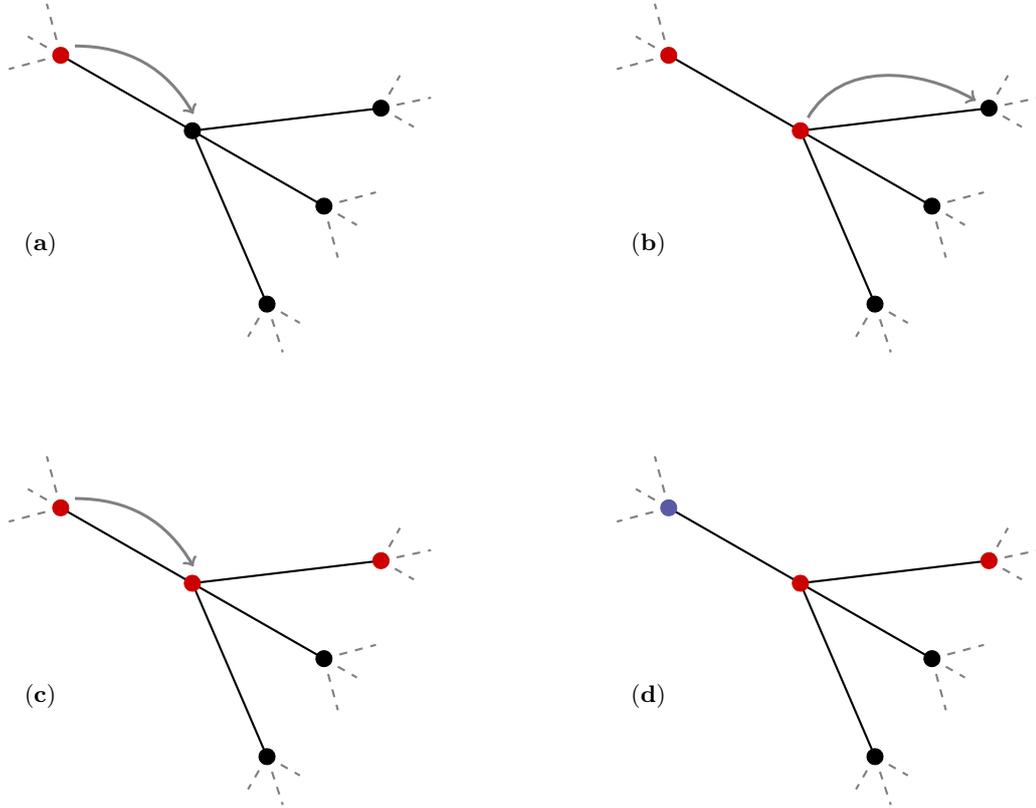

We call the Markov process $(\eta_t)_{t\geq 0}$ the Maki-Thompson rumor model on $\mathbb{T}$, and we abbreviate as MT-model on $\mathbb{T}$. In addition, we refer to the case when $\eta_0({\bf{0}})=1$ and $\eta_0(x)=0$ for all $x\neq {\bf 0}$ as the \textit{standard initial configuration}.     

\smallskip
\begin{definition}\cite[Definition 1]{junior}
Consider the MT-model on $\mathbb{T}$ with the standard initial configuration. We say that there is survival of the rumor if there exist a sequence $\{(v_i,t_i)\}_{i\geq 0}$, with $(v_i,t_i)\in \mathbb{T} \times \mathbb{R}^+$, such that $ v_0={\bf{0}}$, $t_0=0$, $v_{i+1}$ is a successor of $v_i$, $t_i < t_{i+1}$, and $\eta_{t_i}(v_i)=1$, for all $i\geq 0.$ If there is not survival, we say that the rumor becomes extinct. We denote by $\theta(\mathbb{T})$ the survival probability.
\end{definition}

Note that there is survival of the rumor provided we can guarantee the existence of a ray from the root of $\mathbb{T}$ such that all the vertices in the ray were spreaders at some time. The MT-model on an infinite Cayley tree of coordination number $d+1$ was studied in \cite{junior}. In what follows, for simplicity, we let $\theta(d):=\theta(\mathbb{T}_d)$.

\smallskip
\begin{theorem}\label{thm:MTpt}\cite[Theorem 1]{junior}
Consider the MT-model on $\mathbb{T}_d$ with the standard initial configuration. Then $\theta(d)>0$ if, and only if, $d\geq 3$. Moreover, 
$$\theta(d) =  1-  \psi^2,$$
where $\psi$ is the smallest non-negative root of the equation
$$\sum_{i=0}^{d} i! \dbinom{d}{i}\left(\frac{s}{d+1}\right)^i\left(\frac{i+1}{d+1}\right)=s.$$
\end{theorem}

\smallskip

In words, Theorem \ref{thm:MTpt} states that the MT-model exhibit a phase transition. Indeed, for $d=2$, we obtain that the rumor propagates, almost surely, only to a finite set of individuals. In the other cases, for $d\geq 3$, the rumor propagates to infinitely many individuals with positive probability. We point out that although is not considered here, the case $d=1$ may be seen as the path graph with infinite vertices. In that case it is not difficult to see that the number of spreaders will be bounded from above by the sum of two random variables with geometric law. At this point, a curious reader could ask what we can say about more general trees? We address this question by considering the model on random trees. This is the purpose of the next section where we include the main results of this work.

\subsection{The model on a random tree and first results}\label{ss:randomtree}

One way to construct trees randomly is though a the family trees produced by a branching process. In words, we start with one vertex which is the root of the tree and let it connect according to a given discrete distribution; i.e., it has $n$ neighbors with probability $p_n$ for any $n\in\mathbb{N}\cup\{0\}$. Each of these neighbors (if there are any) then connect with new vertices independently with the same law, and so on forever or until some generation goes extinct. These trees are called Galton-Watson trees, or random trees for short. For the sake of simplicity in the exposition, and without loss of generality, our results will be obtained for the MT-model on the tree obtained from a random tree by adding a particular vertex, that we call $s$, connected to its root. We assume that $s$ is the source of the rumor; i.e., the only vertex with the information at time zero. Since $s$ is only connected with the root, then it will spreads the information to it so, from now on, we shall call the standard initial configuration the one for which $\eta_0(s)=\eta_0({\bf 0})=1$, and $\eta_0(x)=0$ for all $x\in V\setminus \{s,{\bf 0}\}$. In addition, as we are interested in to define the model on a tree with infinitely many vertices, we consider super-critical Galton-Watson trees on the event of non-extinction. Notice that this happens provided $\mathbb{E}(\xi)>1$.  In what follows we use the notation $\mathbb{T}_{\xi}$ for such a random tree for which the number of successors of a given vertex $v\neq s$ is given by an independent copy of the non-negative integer valued random variable $\xi$ with $1<\mathbb{E}(\xi)<\infty$. Note that whether $\xi\equiv d$, for $d\geq 2$, we obtain the infinite Cayley tree $\mathbb{T}_d$ with an additional vertex connected to the root. 

We start by characterizing the law of the number of spreaders one spreader generates.

\begin{prop}\label{prop:distxxi}
Let $\xi$ be a non-negative integer valued random variable, and let $(\eta_t)_{t\ge0}$ be the MT-model on $\mathbb{T}_{\xi}$ with the standard initial configuration. If $X(\xi)$ denotes the number of spreaders one spreader generates, then
\begin{equation}\label{eq:xxi}
\mathbb{P}( X(\xi) = i) = (i+1)! \sum_{d=i}^{\infty} \dbinom{d}{i} \left (\frac{1}{d+1} \right )^{i+1} \mathbb{P}(\xi = d).
\end{equation}
Moreover, the generating probability function and the mean of $X(\xi)$ are given, respectively, by
\begin{equation}\label{eq:gflouca}
    G_{\xi}(s) = \sum_{d=0}^{\infty}e^{\frac{d+1}{s}}  \left [ \Gamma \left (d+1,\frac{d+1}{s} \right ) - \frac{d(d+1)}{s}  \Gamma \left (d,\frac{d+1}{s} \right ) \right ] \mathbb{P}( \xi = d) \left (\frac{s}{d+1} \right )^d
\end{equation}
and
\begin{equation}\label{eq:meanxxi}
\mathbb{E}(X(\xi)) = \sum_{d=1}^{\infty} \left [ \frac{e^{d+1} \Gamma (d+1,d+1)}{(d+1)^d}  \mathbb{P}(\xi = d) \right ]- \mathbb{P}(\xi \neq 0),
\end{equation}
where $\Gamma\left(k,x\right):=(k-1)!e^{-x}\displaystyle\sum_{i=0}^{k-1}\frac{x^i}{i!},$ is the incomplete gamma function.
\end{prop}

\begin{proof}

The law of $X(\xi)$ is obtained by noting that 
$$\mathbb{P}( X(\xi) = i)=\sum_{d=0}^{\infty}\mathbb{P}( X(\xi) = i|\xi = d)\mathbb{P}(\xi = d),$$
where $\mathbb{P}( X(\xi) = i|\xi = d)=0$ for $d<i$, and  
$$\mathbb{P}( X(\xi) = i|\xi = d)= (i+1)! \dbinom{d}{i} \left (\frac{1}{d+1} \right )^{i+1},$$
for $d\geq i$, see \cite[Lemma 2]{junior}. In order to obtain the expression \eqref{eq:gflouca} for the generating probability function note that:
\[ G_{\xi}(s) = \sum_{i=0}^{\infty}s^i\mathbb{P}( X(\xi) = i) = \sum_{i=0}^{\infty}s^i(i+1)! \sum_{d=i}^{\infty} \dbinom{d}{i} \left (\frac{1}{d+1} \right )^{i+1} \mathbb{P}(\xi = d)  
\]

Then, by a sequence of suitable arrangement of terms we have, 

$$
\begin{array}{ccl}
G_{\xi}(s) & =&\displaystyle \sum_{d=0}^{\infty} \sum_{i=0}^{d}\frac{(i+1)!}{d+1} \dbinom{d}{i} \left (\frac{s}{d+1} \right )^{i} \mathbb{P}(\xi = d)\\[.5cm] 
& = &\displaystyle \sum_{d=0}^{\infty} \frac{d!\mathbb{P}(\xi = d)}{d+1}  \sum_{i=0}^{d} \frac{(i+1)}{(d-i)!} \left (\frac{s}{d+1} \right )^{i}\\[.5cm] 
& =& \displaystyle\sum_{d=0}^{\infty} \frac{d!\mathbb{P}(\xi = d)}{d+1}  \sum_{i=0}^{d} \frac{(i+1)}{(d-i)!} \left (\frac{s}{d+1} \right )^{i}\\[.5cm]   
& =& \displaystyle\sum_{d=0}^{\infty} \frac{d!\mathbb{P}(\xi = d)}{d+1}  \sum_{j=0}^{d} \frac{(d-j+1)}{j!} \left (\frac{s}{d+1} \right )^{d-j}\\[.5cm] 
& =&\displaystyle \sum_{d=0}^{\infty} \frac{d!\mathbb{P}(\xi = d)}{d+1}\left (\frac{s}{d+1} \right )^{d} \left [ (d+1)\sum_{j=0}^{d} \frac{1}{j!}\left (\frac{d+1}{s} \right )^{j} - \sum_{j=1}^{d} \frac{1}{(j-1)!}\left (\frac{d+1}{s} \right )^{j}  \right ]
\\[.5cm]
& =&\displaystyle \sum_{d=0}^{\infty} \frac{d!\mathbb{P}(\xi = d)}{d+1}\left (\frac{s}{d+1} \right )^{d} \left [ (d+1)\sum_{j=0}^{d} \frac{1}{j!}\left (\frac{d+1}{s} \right )^{j} - \frac{(d+1)}{s}\sum_{l=0}^{d-1} \frac{1}{l!}\left (\frac{d+1}{s} \right )^{l}  \right ]
\end{array}.
$$

Now, consider the incomplete gamma function 

$$\Gamma\left(k,x\right):=(k-1)!e^{-x}\displaystyle\sum_{i=0}^{k-1}\frac{x^i}{i!},$$ 

and note that we can rewrite $G_{\xi}(s)$ as follows.

$$\begin{array}{ccl}
G_{\xi}(s) &=&\displaystyle \sum_{d=0}^{\infty} \frac{d!\mathbb{P}(\xi = d)}{d+1}\left (\frac{s}{d+1} \right )^{d} \left [ \frac{(d+1)e^{\frac{d+1}{s}}}{d!}\Gamma\left(d+1, \frac{d+1}{s}\right) - \frac{(d+1)e^{\frac{d+1}{s}}}{s(d-1)!}\Gamma\left(d, \frac{d+1}{s}\right)\right]\\[.5cm]
& = &\displaystyle \sum_{d=0}^{\infty} \frac{e^{\frac{d+1}{s}} d!\mathbb{P}(\xi = d)}{d+1}\left (\frac{s}{d+1} \right )^{d} \left [ \frac{(d+1)}{d!}\Gamma\left(d+1, \frac{d+1}{s}\right) - \frac{d(d+1)}{s}\Gamma\left(d, \frac{d+1}{s}\right) \right ]\\[.5cm]
& =&\displaystyle \sum_{d=0}^{\infty}e^{\frac{d+1}{s}}  \left [ \Gamma \left (d+1,\frac{d+1}{s} \right ) - \frac{d(d+1)}{s}  \Gamma \left (d,\frac{d+1}{s} \right ) \right ] \mathbb{P}( \xi = d) \left (\frac{s}{d+1} \right )^d
 \end{array}.
 $$

The last equality is \eqref{eq:gflouca}. The mean of $X(\xi)$ is obtained by noting that 
$$\mathbb{E}( X(\xi))=\sum_{d=0}^{\infty}\mathbb{E}( X(\xi)|\xi = d)\mathbb{P}(\xi = d).
$$

To find an expression for $\mathbb{E}(X(\xi)|\xi = d)$ let us give an enumeration for the ignorant neighbors of a spreader with $d+1$ neighbors. In order to do it,  let us write $X(\xi)|\xi = d$ as the sum of indicator random variables $I_1, I_2, \ldots I_d$, where $I_j$ indicates whether the $jth$-ignorant neighbor is informed. Thus, 

\[ \mathbb{E}(X(\xi)|\xi = d) = d\, \mathbb{P}(I_1 = 1) = d\, \mathbb{P} \left( \bigcup_{i=1}^{d}A_i \right) = d\,\sum_{i=1}^{d}\mathbb{P}(A_i)
\]
where $A_i$ is the event of the first neighbor be the $ith$ to be informed by the spreader. Note that the last equality follows from the fact that the events $A_1, A_2, \ldots, A_d$ are $2$ to $2$ mutually exclusive. Then,

$$\mathbb{E}(X(\xi)|\xi = d)= d! \sum_{i=1}^{d} \frac{1}{(d-i)!} \left (\frac{1}{d+1} \right )^{i} = \frac{e^{d+1}\Gamma\left(d+1,d+1\right)}{(d+1)^d} - 1,$$
for $d\geq i$. \end{proof}

We can now state the main theorem of our work.

\begin{theorem}\label{theo:geral}
Let $\xi$ be a non-negative integer valued random variable, and let $(\eta_t)_{t\ge0}$ be the MT-model on $\mathbb{T}_{\xi}$ with the standard initial configuration. Then $\theta(\xi) = 1 - \psi(\xi)$, where $\psi:=\psi(\xi)$ is the smallest non-negative root of the equation
\begin{equation}\label{eq:laly}
G_{\xi}(s)=s.
\end{equation}
Moreover $\theta(\xi)>0$ if, and only if, $\mathbb{E}(X(\xi))>1$.
\end{theorem}

As an application of the above-theorem, whose proof we left to Section \ref{s:proofs}, we consider the MT-model on a random tree produced by random pruning the infinite Cayley tree; i.e., we assume that $\xi$ follows a Binomial distribution of parameters $n$ and $p$.

\begin{corollary}\label{coro:binom}
Let $\xi\sim Binomial(n,p)$, with $n\geq 3$, and let $(\eta_t)_{t\ge0}$ be the MT-model on $\mathbb{T}_{\xi}$ with the standard initial configuration. Then $\theta(n,p) = 1 - \psi$, where $\psi$ is the smallest non-negative root of the equation
\[ \sum_{d=0}^{n} \displaystyle \left[\frac{1}{d+1} \dbinom{n}{d} p^d(1-p)^{n-d} \sum_{i=0}^{d} (i+1)! \dbinom{d}{i} \left(\frac{s}{d+1} \right)^i \right]= s.\]
Moreover $\theta(n,p)>0$ if, and only if,

\begin{equation}\label{eq:corobinom}
    \sum_{d=1}^n \frac{1}{(n-d)!}\left(\frac{p}{1-p}\right)^d\sum_{j=1}^d\frac{1}{(d+1)^j(d-j)!}>\dfrac{1}{(1-p)^n n!}.
\end{equation}
\end{corollary}

\begin{proof} Let $ \xi \sim Binomial(n,p) $. Proposition \ref{prop:distxxi} give us an expression for the generating probability function $G_{\xi}(s)$ using the incomplete Gamma function. Since here we are considering a simple law for $\xi$ it is enough to note that

$$G_{\xi}(s)=\displaystyle\sum_{d=0}^{\infty} \frac{d!\mathbb{P}(\xi = d)}{d+1}  \sum_{i=0}^{d} \frac{(i+1)}{(d-i)!} \left (\frac{s}{d+1} \right )^{i},$$

\noindent
so we can obtain a version of \eqref{eq:laly} only by replacing the law of $\xi$ in the previous expression. Now, by Theorem \ref{theo:geral}, we have $\theta(n,p)>0$ if, and only if, $\mathbb{E}(X(\xi))>1.$ Note that 
\begin{align}\label{eq:esp}
\mathbb{E}(X(\xi))&=\sum_{d=1}^n \mathbb{E}(X|\xi=d)\mathbb{P}(\xi=d)\nonumber\\
&= \sum_{d=1}^n d!\sum_{j=1}^d\frac{1}{(d+1)^j(d-j)!}\binom{n}{d}p^d(1-p)^{n-p}\nonumber\\
&=\sum_{d=1}^n d!\binom{n}{d}p^d(1-p)^{n-d}\sum_{j=1}^d\frac{1}{(d+1)^j(d-j)!}\nonumber\\
&=(1-p)^n \sum_{d=1}^n \frac{d!n!}{d!(n-d)!}\left(\frac{p}{1-p}\right)^d\sum_{j=1}^d\frac{1}{(d+1)^j(d-j)!}\nonumber\\
&=(1-p)^n n! \sum_{d=1}^n \frac{1}{(n-d)!}\left(\frac{p}{1-p}\right)^d\sum_{j=1}^d\frac{1}{(d+1)^j(d-j)!}.\nonumber\\
\end{align}
Then  $\theta(n,p)>0$ if, and only if,
$$
(1-p)^n n! \sum_{d=1}^n \frac{1}{(n-d)!}\left(\frac{p}{1-p}\right)^d\sum_{j=1}^d\frac{1}{(d+1)^j(d-j)!}>1
$$

\noindent
so we get \eqref{eq:corobinom} and this complete the proof.
\end{proof}

Corollary \ref{coro:binom} gains in interest if we realize that it give us a way to find the value of $p$ for which the behavior of the model moves from extinction to survival of the rumor provided $n$ is fixed. Moreover, it is not difficult to see that we recover Theorem \ref{thm:MTpt} by letting $n=d$ and $p=1$. Now, let us consider an example to illustrate what happens for a tree for which $\mathbb{E}(\xi)$ is between $2$ and $3$.

\begin{exa}\label{exa:binom3p}Let $\xi\sim Binomial(3,p)$ and consider the MT-model on $\mathbb{T}_{\xi}$ with the standard initial configuration. Our interest is for $p\in(2/3,1)$; that is $\mathbb{E}(\xi)\in (2,3)$. Indeed, if $p<2/3$ we known that the rumor becomes extinct almost surely (by comparison with Theorem \ref{thm:MTpt} for $d=2$). On the other hand, for $p=1$ the rumor survives with positive probability (see Theorem \ref{thm:MTpt} for $d=3$). By Corollary \ref{coro:binom} we have $\theta(3,p) = 1 - \psi$ where $\psi$ is the smallest non-negative root of the equation
\[ (1-p)^3 +\frac{3}{2}(p-1)^2 p(s+1) - \frac{1}{3}(p-1)p^2(2s^2+4s+3) + \frac{1}{32} p^3(3s^3 +9s^2 +12s +8) = s.
\]
Thus $\theta(3,p) > 0$ if, and only if, $5p^3 -32p^2 +144p -96 >0$. This in turns implies that
\begin{equation}
\theta(3,p)=\left\{
\begin{array}{cl}
 \displaystyle\frac{-5p^2 +32 p -2\sqrt{p(-5p^3 -8p^2 -68p+216)}}{9p^2},&\text{ if }p > p_c,\\[.5cm]
 0,&\text{ if }p \leq p_c,
\end{array}\right.
\end{equation}
where
$$ p_c := \frac{2}{15} \left(16 - (142 \times 2^{\frac{2}{3}})/(45 \sqrt{5689} - 2407)^{\frac{1}{3}} + (2 (45 \sqrt{5689} - 2407))^{\frac{1}{3}} \right)\approx0.78753.
$$
Moreover, note that $\max_{p \in [0,1]} \theta(3,p) = 3- \frac{2}{3}\sqrt{15} \approx 0.418.$ See Figure \ref{fig:ptbinom}.
\end{exa}

\bigskip
	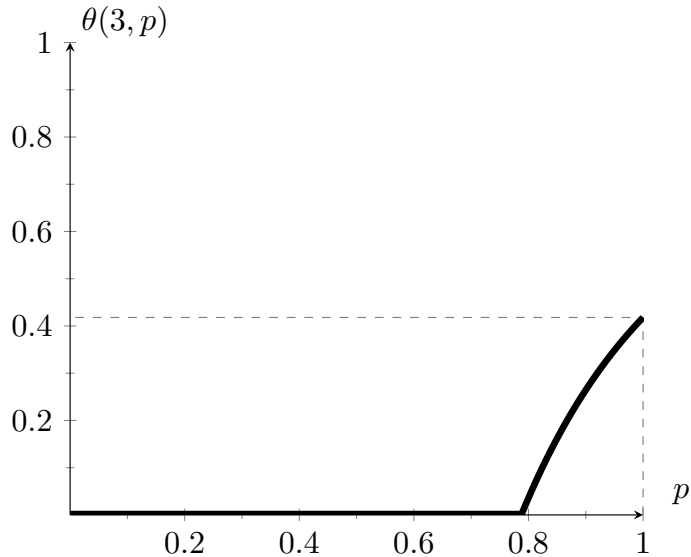
\begin{figure}[!htb]
	\pgfplotsset{my style/.append style={axis x line=middle, axis y line=
middle, xlabel={$p$}, ylabel={$\theta(3,p)$}}}
	\begin{tikzpicture}[scale=1.1]
\begin{axis}[my style,
xmin=0, xmax=1, ymin=0, ymax=1, minor tick num=1]
 \addplot +[mark=none,gray, dashed] coordinates {(1,0) (1, 0.418)(0,0.418)};
\addplot[domain=0:0.78753, line width=1mm] {0};
\addplot[domain=0.78753:1, line width=0.8mm] {(-5*x^2+32*x-2*sqrt(x*(-5*x^3-8*x^2-68*x+216)))/(9*x^2)};
\end{axis}
\end{tikzpicture}

		\caption{Survival probability for the MT-model on $\mathbb{T}_{\xi}$ with $\xi\sim Binomial(3,p)$ and the standard initial configuration. For $p\leq 2/3$ we known that $\theta(3,p)=0$ as a consequence of the same behavior for the infinite Cayley tree of coordination number $3$, see Theorem \ref{thm:MTpt} for $d=2$, and a coupling argument. Corollary \ref{coro:binom} allows to understand the behavior of the model for $p>2/3$ localizing the critical parameter in $\approx0.78753.$}\label{fig:ptbinom}
	\end{figure}
	
	\subsection{Further examples and properties}

Of course, as a consequence of Theorem \ref{theo:geral}, different distributions give rise to different versions of Corollary \ref{coro:binom}. However, the exact localization of the critical parameter is a task that can be difficult depending of the considered law. In those cases one of the usual strategies is the computation of non-trivial bounds for such a value. Let us illustrate it with some examples. 

\begin{exa}\label{exa:poisson}
Let $\xi\sim Poisson(\lambda)$, with $\lambda > 1$, and let $(\eta_t)_{t\ge0}$ be the MT-model on $\mathbb{T}_{\xi}$ with the standard initial configuration. A direct application of Theorem \ref{theo:geral} allow us to conclude that there exist a critical value $\lambda_c \in (2.49782; 2.49785)$ such that $\theta(\lambda)>0$ if, and only if, $\lambda > \lambda_c$. Indeed, note that Equation \eqref{eq:meanxxi} from Proposition \ref{prop:distxxi} leads to
$$
g(\lambda) :=\mathbb{E}(X(\xi))= e^{-\lambda +1}\sum_{d=1}^{\infty} \left [ \frac{(\lambda e)^d \Gamma (d+1,d+1)}{(d+1)^d d!} \right] +e^{-\lambda}-1.
$$
On the one hand we have
\begin{equation}\label{eq:poisson1}
g(\lambda) < e^{-\lambda}-1  + e^{-\lambda +1}\sum_{d=1}^{17} \left [ \frac{(\lambda e)^d \Gamma (d+1,d+1)}{(d+1)^d d!} \right] + e^{-\lambda +1}\sum_{d=18}^{\infty} \left [ \frac{(\lambda e)^d}{(d+1)^d} \right],
\end{equation}
while, on the other hand we obtain
\begin{equation}\label{eq:poisson2}
g(\lambda) >  e^{-\lambda}-1  + e^{-\lambda +1}\sum_{d=1}^{17} \left [ \frac{(\lambda e)^d \Gamma (d+1,d+1)}{(d+1)^d d!} \right].
\end{equation}
Finally, by taking $\lambda=2.49782$ in \eqref{eq:poisson1} the right side is equal to
$$e^{-2.49782}-1  + e^{-1.49782}\sum_{d=1}^{17} \left [ \frac{(2.49782 e)^d \Gamma (d+1,d+1)}{(d+1)^d d!} \right] + e^{-1.49782}\sum_{d=18}^{\infty} \left [ \frac{(2.49782e)^d}{(d+1)^d} \right]
$$
so $g(2.49782) < 0.999997 + 2.332.10^{-9} < 1.$ Similarly, by taking $\lambda = 2.49785$ in \eqref{eq:poisson2} we get
$$  g(2.49785) >  e^{-2.49785}-1  + e^{-2.49785}\sum_{d=1}^{17} \left [ \frac{(2.49785 e)^d \Gamma (d+1,d+1)}{(d+1)^d d!} \right] \geq 1.00001 > 1.
$$
Therefore we conclude the result from Theorem \ref{theo:geral}.
\end{exa}

\begin{exa}\label{exa:geom}
Let $\xi\sim Geom(p)$, and let $(\eta_t)_{t\ge0}$ be the MT-model on $\mathbb{T}_{\xi}$ with the standard initial configuration. Then, there exist a critical value $p_c\in (0.25894,  0.25895)$ such that $\theta(p)>0$ if, and only if, $p < p_c$. In this case, Equation \eqref{eq:meanxxi} from Proposition \ref{prop:distxxi} leads to
\begin{equation*}
g(p)=ep\sum_{d=1}^\infty \Big(\frac{(1-p)e}{d+1}\Big)^d\Gamma(d+1,d+1)-(1-p).
\end{equation*}
On the one hand we have
\begin{equation}\label{eq:geom1}
g(p) > ep\sum_{d=1}^{300} \Big(\frac{(1-p)e}{d+1}\Big)^d\Gamma(d+1,d+1)-(1-p).
\end{equation}
while, on the other hand we obtain
\begin{equation}\label{eq:geom2}
g(p) < ep\sum_{d=1}^{300} \Big(\frac{(1-p)e}{d+1}\Big)^d\Gamma(d+1,d+1) + ep\sum_{d=301}^{\infty} \left [ \left ( \frac{(1-p)e}{2.2} \right )^d  \right ]-(1-p).
\end{equation}
which follows from the following sequence of inequalities:
$$
\begin{array}{ccl}
\displaystyle\sum_{d=301}^{\infty} \left(\frac{(1-p)e}{d+1}\right)^d\Gamma(d+1,d+1)& \leq&\displaystyle \displaystyle\sum_{d=301}^{\infty} \left(\frac{(1-p)e}{d+1}\right)^d d! \\[.5cm]
&\leq& \displaystyle\sum_{d=301}^{\infty} \left(\frac{(1-p)ed}{2.2(d+1)}\right)^d \frac{d!}{(\frac{d}{2.2})^d}\\[.5cm]
&\leq&\displaystyle \sum_{d=301}^{\infty} \left(\frac{(1-p)ed}{2.2(d+1)}\right)^d\\[.5cm]  &\leq& \displaystyle\sum_{d=301}^{\infty} \left(\frac{(1-p)e}{2.2}\right)^d
\end{array}
$$
By taking $p=0.25894$ in \eqref{eq:geom1} we obtain
$$ g(0.25894)  > 0.25894 e \sum_{d=301}^{\infty} \left(\frac{0.74106ed}{2.2(d+1)}\right)^d \Gamma(d+1,d+1) - 0.74106 = 1.00003 > 1.
$$
By taking $p=0.25895$ in \eqref{eq:geom2} we have
$$g(0.25895) < 0.25895 e \sum_{d=1}^{300} \left(\frac{0.74105e}{d+1}\right)^d\Gamma(d+1,d+1) + \sum_{d=301}^{\infty} \left [ \left ( \frac{0.74105e}{2.2} \right )^d  \right ]-0.74105
$$
and this in turns implies
$$
g(0.25895) < 0.999993 + 2.50463 \times 10^{-11} < 1.
$$
Thus, the result is a consequence of Theorem \ref{theo:geral}.

\end{exa}


In what follows we discuss an alternative strategy when a direct application of Theorem \ref{theo:geral}, to a given law of $\xi$, is not possible.

\begin{prop}\label{prop:mu}
Let $(\xi_n)_{n\in\mathbb{N}}$ and $\xi$ be non-negative integer valued random variables, and assume that $\lim_{n\to \infty} \xi_n =\xi$, in law. Let $(\eta^n_t)_{t\ge0}$ and $(\eta_t)_{t\ge0}$ be MT-models on $\mathbb{T}_{\xi_n}$ and $\mathbb{T}_{\xi}$, respectively, with the standard initial configuration. If $\mu_n$ and $\mu$ denote the expected value of the number of spreaders one spreader generates for each one of these models, then $\lim_{n \to \infty} \mu_n >1$ if, and only if, $\mu>1$.
\end{prop}

\begin{proof}
Since $\mu_n:=\mathbb{E}(X(\xi_n))$ and $\mu:=\mathbb{E}(X(\xi))$, then by \eqref{eq:meanxxi} we have
$$\mu_{n} = \sum_{d=1}^{\infty} \left [ \frac{e^{d+1} \Gamma (d+1,d+1)}{(d+1)^d}  \mathbb{P}(\xi_n = d) \right] - \mathbb{P}(\xi_n \neq 0)$$
and
$$\mu = \sum_{d=1}^{\infty} \left [ \frac{e^{d+1} \Gamma (d+1,d+1)}{(d+1)^d}  \mathbb{P}(\xi = d) \right ]- \mathbb{P}(\xi \neq 0).$$

\noindent
But
$$
\begin{array}{ccl}
\displaystyle \lim_{n \to \infty}\mu_{n} &=&\displaystyle \lim_{n \to \infty} \left \{ \sum_{d=1}^{\infty} \left [ \frac{e^{d+1} \Gamma (d+1,d+1)}{(d+1)^d}  \mathbb{P}(\xi_n = d) \right ] - \mathbb{P}(\xi_n \neq 0) \right \}\\[.5cm] 
&=&\displaystyle \sum_{d=1}^{\infty} \left [ \frac{e^{d+1} \Gamma (d+1,d+1)}{(d+1)^d}  \mathbb{P}(\xi = d) \right ] - \mathbb{P}(\xi \neq 0)\\[.5cm] 
&=& \mu.
\end{array}
$$

We point out that the interchange of limit and summation is guarantee by the Dominated Convergence Theorem, see \cite[Theorem 9,1, p. 26]{thorison}.

\end{proof}

\begin{exa}\label{cor:binompoi}
Let $\xi_n\sim Binomial(n,p)$, for any $n\geq 3$, and let $(\eta_t)_{t\ge0}$ be the MT-model on $\mathbb{T}_{\xi}$ with the standard initial configuration. Consider for any $n\geq 3$ the critical parameter:   
\begin{equation}
    p_c (n) = \inf\{p \in [0,1]: \theta(n,p) > 0\}.
\end{equation}
If we assume $\lim_{n\to \infty}np=\lambda$ and we let $\xi\sim Poisson(\lambda)$, then Proposition \ref{prop:mu} applies and therefore $\lim_{n \to \infty} n p_c(n) \approx 2.4978$, where the limit is obtained from Example \ref{exa:poisson}. See Table \ref{tab:my_label}.

\end{exa}

\begin{table}[h!]
    \centering
\begin{tabular}{c|c|c|c|c|c|c|c|c}

  $n$ & $3$ & $4$ & $5$ & $10$ & $25$ & $50$ & $100$ & $150$\\\hline
  $p_c(n)\approx$ & $ 0.78753$ & $0.599322 $ & $0.483563$  & $0.24582$ & $0.09928$  & $0.04979$ & $0.02495$ & $0.01663$ \\\hline
  $np_c(n)\approx$ & $2.362591$ & $2.397288$ & $2.417815$ & $2.4582$ & $2.48200$ & $2.48950$ & $2.49390$ & $2.49510$
  \end{tabular} 
    \caption{Illustration of the behavior of the critical parameter of Example \ref{cor:binompoi}.}
    \label{tab:my_label}
\end{table}

We finish the section with a similar analysis for the survival probabilities.

\begin{theorem}\label{teo:lim}
Let $(\xi_n)_{n\in\mathbb{N}}$ be a non-decreasing sequence of non-negative integer valued random variables of vectors parameters given by $(\nu_n)_{n\in\mathbb{N}}$. Let $\xi$ be another non-negative integer valued random variable, of vector parameters $\nu$, and assume that 
$$
\lim_{n\to \infty} \xi_n =\xi,\text{ in law}.
$$
Let $(\eta^n_t)_{t\ge0}$ and $(\eta_t)_{t\ge0}$ be MT-models on $\mathbb{T}_{\xi_n}$, for $n\in\mathbb{N}$, and $\mathbb{T}_{\xi}$, with the standard initial configuration and survival probabilities given by $\theta(\nu_n)$ and $\theta(\nu)$, respectively. Then 
\begin{equation}
    \lim_{n \to \infty} \theta(\nu_n) = \theta(\nu).
\end{equation}
\end{theorem}

\begin{proof}
Let $x$ a non-negative integer and note that

$$\mathbb{P}( X(\xi) \leq x | \xi = d) = 1 -\prod_{i=1}^{x} \left( 1- \frac{i}{d}\right) \geq 1 -\prod_{i=1}^{x} \left( 1- \frac{i}{d+1}\right)
 = \mathbb{P}( X(\xi) \leq x | \xi = d+1). 
$$
In general, for any $x$, we can prove by induction that $d_1 < d_2$ implies
$$\mathbb{P}( X \leq x | \xi = d_1)  \geq \mathbb{P}( X \leq x | \xi = d_2)$$
On the other hand, if $\xi_n$ converges in law to $\xi$, then the 
Dominated Convergence Theorem, see \cite[Theorem 9,1, p. 26]{thorison}, allows the interchange of limit and summation, and so:

$$
\begin{array}{ccl}
\displaystyle\lim_{ n \to \infty} \mathbb{P}(X(\xi_n) = i) &=& \displaystyle\lim_{ n \to \infty} (i+1)! \sum_{d=i}^{\infty} \dbinom{d}{i} \left (\frac{1}{d+1} \right )^i \mathbb{P}(\xi_n = d)\\[.5cm]
&=& (i+1)! \displaystyle\sum_{d=i}^{\infty} \dbinom{d}{i} \left (\frac{1}{d+1} \right )^i \displaystyle\lim_{n \to \infty} \mathbb{P}(\xi_n = d)\\[.5cm] 
& =& (i+1)! \displaystyle\sum_{d=i}^{\infty} \dbinom{d}{i} \left (\frac{1}{d+1} \right )^i \mathbb{P}(\xi = d)\\[.5cm] 
&=& \mathbb{P}(X(\xi)= i).
\end{array}
$$
That is, $X(\xi_n) \to X(\xi)$ in law. Finally, it is not difficult to see that $X(\xi_n)$ is stochastically dominated by $X(\xi)$ provided $\xi_n$ be stochastically dominated by $\xi$. See \cite[Equation (3.1), page 4]{thorison} for more details. Therefore, by applying \cite[Proposição 4.2]{FAV} we get 
$$ \lim_{n \to \infty} \theta_n (\nu) = \theta (\nu).
$$

\end{proof}

\subsection{The range of the rumor on a random tree}

As mentioned in the Introduction, one way to measure the ``size'' of the rumor on a finite graph is through the analysis of the remaining proportion of ignorants at the end of the process. In the case of a tree we do it through the notion of range of the rumor, which is the longest distance, from the root, at which there is a non-ignorant individual.    
\begin{theorem}\label{theo:range}
Let $\xi$ be a non-negative integer valued random variable, and let $(\eta_t)_{t\ge0}$ be the MT-model on $\mathbb{T}_{\xi}$ with the standard initial configuration. Consider $X(\xi)$, and let $\mu_{\xi}$ and $G_{\xi}$ be its mean and its p.g.f, respectively, as defined by \eqref{eq:xxi}. Let 
\begin{equation}\label{eq:range} 
R(\xi):=\max\{n\geq 1: \eta_{t}(x)=1 \text{ for some }x\in \partial \mathbb{T}_{\xi,n}, \text{ and } t\in \mathbb{R}^+\},
\end{equation} 
be the range of the rumor. If $\mu_{\xi} G_{\xi}''(s) - G_{\xi}'(s)G_{\xi}''(1) \geq 0$, then 

\begin{equation}\label{eq:relouca} \frac{ (\ell-1) \mu_{\xi}^{m+1}}{\ell - \mu_{\xi}^{m+1}} \leq \mathbb{P}(R(\xi) > m) \leq \frac{ (u-1) \mu_{\xi}^{m+1}}{u - \mu_{\xi}^{m+1}},\end{equation}

\noindent
for any $m\geq 0$, where 

$$u := \left\{\mu_{\xi} \,\mathbb{E} \left(\frac{1}{\xi+1} \right)\right\}\left\{\mu_{\xi} + \mathbb{E} \left(\frac{1}{\xi+1} \right)  - 1\right\}^{-1},$$

\noindent
and

$$\ell := 1 - \left\{\frac{2}{G_{\xi}''(1)}\right\}\,\mu_{\xi}\, (\mu_{\xi}-1).$$

Moreover,

\begin{equation}\label{eq:relouca2}
    (u-1) \sum_{m=0}^{\infty} \frac{\mu_{\xi}^{m+1}}{u - \mu_{\xi}^{m+1}} \leq \mathbb{E}(R(\xi)) \leq (\ell-1) \sum_{m=0}^{\infty} \frac{\mu_{\xi}^{m+1}}{\ell - \mu_{\xi}^{m+1}}.
\end{equation}
 \end{theorem}

The proof of the above-theorem is left to Section \ref{s:proofs}. Let us illustrate the applicability of Theorem \ref{theo:range} with an example.

\begin{exa}\label{exe:lala}
Let $\xi\sim Binomial(3,p)$ and let $(\eta_t)_{t\ge0}$ be the MT-model on $\mathbb{T}_{\xi}$ with the standard initial configuration. In this case,

$$
\begin{array}{ccl}
G_{\xi}(s) &=&\displaystyle \frac{3p^3}{32}s^3 + \left( \frac{2p^2}{3} - \frac{37p^3}{96} \right)s^2 + \left( \frac{13p^3}{24} -\frac{5p^2}{3} +\frac{3p}{2} \right )s + \left ( \frac{-p^3}{4} + p^2 -\frac{3p}{2} +1 \right),\\[.4cm]
G'(s) &=&\displaystyle \frac{9p^3}{32}s^2 + \left ( \frac{4p^2}{3} - \frac{74p^3}{96}  \right )s + \left( \frac{13p^3}{24} -\frac{5p^2}{3} +\frac{3p}{2} \right ),\\[.4cm]
G''(s) &=&\displaystyle \frac{9p^3}{32}s + \left ( \frac{4p^2}{3} - \frac{74p^3}{96}  \right ),\\[.2cm]
\end{array}
$$

\noindent
which lead to
$$
\begin{array}{ccl}
G'(1) &=& \displaystyle\frac{5p^3}{96}  - \frac{p^2}{3} + \frac{3p}{2},\\[.3cm]  G''(1) &=& \displaystyle\frac{4p^2}{3} - \frac{5p^3}{24},\\[.3cm] 
G(0)&=&\displaystyle \frac{-p^3}{4} +p^2 - \frac{3p}{2} +1.
\end{array}
$$

Note that

$$G'(1)G''(s) - G'(s)G''(1) \geq 0\,\,\, \text{ if }\,\,\, 335p^2 -3424p+4304 >0,$$

\noindent
which holds for any $p \in [0,1]$. Therefore, we can apply Theorem \ref{theo:range} by noting that

$$\mu_{\xi} =\displaystyle \frac{p(5p^2 -32p+144)}{96},$$ 

$$\ell =\displaystyle \frac{-25p^5 + 320p^4 -2464p^3 +8736p^2 -17664p+13824}{192p(32-5p)},$$

$$u =\displaystyle \frac{(2-p)(p^4 -6p^3 +16p^2 -20p+12)(5p^2-32p+144)}{4p(64-19p)(p^2-4p+6)}.$$

\smallskip
\noindent
See Figure \ref{fig:range} for an illustration of the resulting bounds of $\mathbb{E}(R(\xi))$.

\end{exa}

	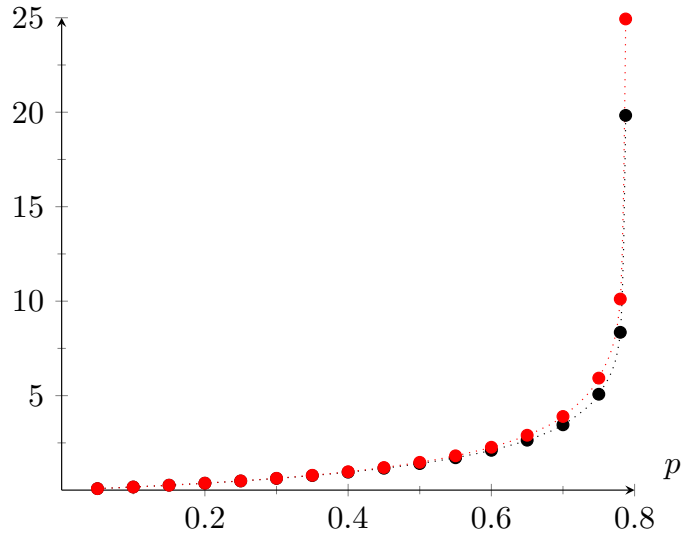
\begin{figure}[!htb]
	\pgfplotsset{my style/.append style={axis x line=middle, axis y line=
middle, xlabel={$p$}, ylabel={}}}
	\begin{tikzpicture}[scale=1.1]
\begin{axis}[my style,
xmin=0, xmax=0.8, ymin=0, ymax=25, minor tick num=1]
\addplot[
    color=black,
dotted,
mark=*,
mark options={solid},
smooth
    ]
    coordinates {
    (0.05,0.078)(0.1,0.164)(0.15,0.259)(0.2,0.365)(0.25,0.483)(0.3,0.617)(0.35,0.770)(0.4,0.948)(0.45,1.156)(0.5,1.405)(0.55,1.712)(0.6,2.105)(0.65,2.641)(0.7,3.456)(0.75,5.074)(0.78,8.351)(0.7875,19.832)
    };
    
    \addplot[
    color=red,
dotted,
mark=*,
mark options={solid},
smooth
    ]
    coordinates {
    (0.05,0.078)(0.1,0.164)(0.15,0.26)(0.2,0.367)(0.25,0.488)(0.3,0.626)(0.35,0.785)(0.4,0.973)(0.45,1.197)(0.5,1.471)(0.55,1.816)(0.6,2.269)(0.65,2.904)(0.7,3.898)(0.75,5.928)(0.78,10.113)(0.7875,24.936)
    };
\end{axis}
\end{tikzpicture}

		\caption{Comparison between the bounds obtained in Theorem \ref{theo:range} for the range of spreading in the MT-model on $\mathbb{T}_{\xi}$ with $\xi\sim Binomial(3,p)$, see Example \ref{exe:lala}. For some values of $p$ the lower bound of $E(R(\xi))$ is represented by a black dot while the upper bound is represented by a red dot.}\label{fig:range}
	\end{figure}

\section{Proof of Theorems \ref{theo:geral} and \ref{theo:range}}\label{s:proofs}

\subsection{The rumor model seen as a branching process}
The main idea behind the proofs of Theorems \ref{theo:geral} and \ref{theo:range} is the identification of an underlying branching process related to the rumor model. After doing that we can apply well-known results of the theory of Branching Processes. This approach has been used before in \cite{junior} to study the MT-model on infinite Cayley trees. For a deeper discussion of branching processes we refer the reader to \cite[Chapter 2]{schinazi}. 
 
Consider a non-negative integer valued random variable $\xi$ and the MT-model on $\mathbb{T}_{\xi}$, $(\eta_t)_{t\ge0}$, with the standard initial configuration. For any $n\geq 0$ we let
$$\mathcal{Z}_{n}(\xi):=\left\{v\in \partial \mathbb{T}_{\xi,n+1}: \bigcup_{t> 0} \{\eta_{t}(v)=1\right\},$$
and we define the random variable $Z_n:= |\mathcal{Z}_n(\xi)|$. Thus defined, $\mathcal{Z}_{0}(\xi)$ is formed by those vertices at distance one from ${\bf 0}$ which are spreaders at some time, $\mathcal{Z}_{1}(\xi)$ is formed by those vertices at distance two from ${\bf 0}$ which are spreaders at some time, and so on. Moreover, notice that the number of spreaders one spreader generates behaves according to a independent copy of a nonnegative integer valued  random variable $X(\xi)$. Then $Z_0$ is equal to $X(\xi)$ in law, and 
\begin{equation}\label{eq:BPZ}
Z_{n+1}(\xi)=\sum_{i=1}^{Z_n(\xi)}X_i(\xi),
\end{equation}
where $X_1(\xi),X_2(\xi),\ldots$ are independent copies of $X(\xi)$. As a direct consequence of the construction above we obtain the following result which is the key to prove our theorems.

\begin{lemma}\label{lem:bp}
The stochastic process $(Z_n(\xi))_{n\geq 0}$ defined by \eqref{eq:BPZ} is a branching process with offspring's distribution given by the law of $X(\xi)$.
\end{lemma}


\subsection{Proof of Theorem \ref{theo:geral}} Let $\xi$ be a non-negative integer valued random variable, and let  $(\eta_t)_{t\ge0}$ be the MT-model on $\mathbb{T}_{\xi}$ with the standard initial configuration. By Lemma \ref{lem:bp} we have that the rumor survives if, and only, if, the branching process with offspring's distribution given by the law of $X(\xi)$ survives. Therefore, $\psi(\xi)$ coincides with the extinction probability for such a branching process, which is the smallest non-negative root of the equation $G_{\xi}(s)=s.$ Moreover, also as a consequence of well-known result of branching process we have that $\psi(\xi)<1$ if, and only if, $\mathbb{E}(X(\xi))>1$. See \cite[Theorem 1.1, p. 19]{schinazi} for more details.

\subsection{Proof of Theorem \ref{theo:range}} Let $\xi$ be a non-negative integer valued random variable, and let  $(\eta_t)_{t\ge0}$ be the MT-model on $\mathbb{T}_{\xi}$ with the standard initial configuration. Thanks to the connection between our model with a branching process, stated in Lemma \ref{lem:bp}, it is enough to point out that the range of the rumor defined in \eqref{eq:range}, plus $1$, it is the same than the extinction time of the associated branching processes. Therefore we appeal to arguments of \cite{agresti} for deriving bounds for the tail of such a random time.  The main approach of \cite{agresti} consists of reducing the problem of deriving these bounds to a problem involving the analysis of a family of $p.g.f.$'s whose iterates can be calculated; i.e., the fractional linear generating functions (f.l.g.f.). We refer the reader to \cite[Definition 2]{wang}, for a suitable definition of a fractional linear generating function. Our first task is to obtain two f.l.g.f. $f_l(s)$ and $f_u(s)$ such that $f_l(s) \leq G_{\xi}(s) \leq f_u(s), 0 \leq s \leq 1$. In order to do it we apply \cite[Theorem, p. 450]{wang}, verifying its corollary of p. 451. Then, $f_l(s)$ and $f_u(s)$ are such that
\[ c_l = \frac{G_{\xi}''(1)}{2G_{\xi}'(1) + G_{\xi}''(1)} = \frac{G_{\xi}''(1)}{2\mu_{\xi} + G_{\xi}''(1)} 
\]
and
\[ c_u = \frac{G_{\xi}'(1) + G_{\xi}(0) -1}{G_{\xi}'(1)} = \frac{\mu_{\xi} + \mathbb{E} \left(\frac{1}{\xi+1} \right) - 1}{\mu_{\xi}}.
\]
Note that the inequalities hold also for the $m$-th compositions of these functions. Thus, 
\[ f_l^m(0) \leq  G_{\xi}^m(0)  \leq f_u^m(0).
\]
and this in turns implies
\[ f_l^m(0) \leq  \mathbb{P}(R(\xi) \leq m+1) \leq f_u^m(0).
\]

Now, the inequalities in \eqref{eq:relouca} follow by \cite[Eq. (3.1), p. 330]{agresti}, with

\[ \ell =  \frac{2G_{\xi}'(1) + G_{\xi}''(1) - 2 [G_{\xi}'(1)]^2}{G_{\xi}''(1)}  = 1 - \left\{\frac{2}{G_{\xi}''(1)}\right\}\,\mu_{\xi}\, (\mu_{\xi}-1)
\]
and
\[ u =  \frac{G_{\xi}(0)G_{\xi}'(1)}{G_{\xi}'(1) + G_{\xi} (0)-1} = \left\{\mu_{\xi} \,\mathbb{E} \left(\frac{1}{\xi+1} \right)\right\}\left\{\mu_{\xi} + \mathbb{E} \left(\frac{1}{\xi+1} \right)  - 1\right\}^{-1}.
\]
Finally, \eqref{eq:relouca2} is a direct consequence of \eqref{eq:relouca}.

\bigskip
\section{Discussion}

The theory of mathematical models for rumor spreading has increased in the last years. In this work we contributed with this field by advancing our understanding of the behavior of the Maki-Thompson rumor model on a random tree. Our approach rely on the comparison of the rumor model with a suitable defined branching process. This allows to obtain results regarding a phase-transition for the model, in the sense of the existence of a critical value around which the rumor either becomes extinct almost-surely or survives with positive probability. Our arguments also allow to obtain an approximation for the mean range of the rumor. The analysis proposed here can be adapted to study the model in other tree-like graphs. In addition, we emphasize that an exploration on the connection between our approach and the one developed recently by \cite{gleeson} for the description of information cascades on Twitter may be an issue of interesting research. Below, we suggest other two directions for further research. 

\subsection{The role of the underlying distribution for the localization of the critical point}

In this work we consider locally finite rooted trees with infinitely many vertices. It is more intuitive to imagine the tree as growing away from its root, in the sense that each vertex has branches leading to its successors, which are its neighbors that are further from the root. With the aim of assign an average branching number to an arbitrary infinite locally finite tree it was introduced in \cite{lyons} the so called branching number of a tree. Given a tree $\mathbb{T}$, its branching number is denoted by $br\left(\mathbb{T}\right)$ and may be seen as the exponential of the Hausdorff dimension of its boundary. The formal definition of the branching number uses the concept of flows on trees but we do not develop this point here since it exceeds the scope of the paper. A more complete theory may be obtained from \cite[Chapter 3]{lyons}. The important point to note here is that if $\mathbb{T}_{\xi}$ is a Galton-Watson tree with mean $\mathbb{E}(\xi)>1$, then $br\left(\mathbb{T}_{\xi}\right)= \mathbb{E}(\xi)$ a.s. given nonextinction. See \cite[Corollary 5.10]{lyons} for more details and note that, as a consequence, $br\left(\mathbb{T}_d\right)=d$. 

One could suspect that the branching number of a given tree has an important role for the localization of the critical parameter. More precisely, is there a critical branching number around which the rumor either becomes extinct almost-surely or survives with positive probability? With this question in mind, the known results for $\mathbb{T}_d$, for $d\geq 2$, would localize such a value between $2$ and $3$. In Example \ref{exa:binom3p} we exhibit a family of random trees $\mathbb{T}_{\xi}$ for which the rumor survives with positive probability, if and only if, $br\left(\mathbb{T}_{\xi}\right)> 2.36259$. However, in Examples \ref{exa:poisson} and \ref{exa:geom} we have a different picture. While on one hand, Example \ref{exa:poisson} provides a family of random trees for which the critical mean is approximately $2.4978$, on the other hand Example \ref{exa:geom} exhibit a family of random trees for which the critical mean is localized approximately in $2.8625$. So the question that remains is how the underlying distribution generating the tree affect the localization of the critical point. 

\subsection{The number of spreaders one spreader generates and the coupon collector´s problem}

As pointed out in \cite[Remark 1]{junior}, there exist an interesting connection between the random variables arising in our model and some variables coming from the Coupon Collector’s Problem. The problem is well-known from the Probability literature and it is usually stated as the problem of a collector who obtains, at each stage, a coupon which is
equally likely to be any one of $n$ types. Assuming that the results of successive stages are independent, among other results, one could have interest into the first stage at which all $n$ coupons have been
picked at least once. This problem gave rise to a series of generalizations which have been addressed in the literature, see for example \cite{boneh}. Also, the underlying process has been used to study a related information transmission model, see \cite{comets}. In our model, note that the number of of spreaders one spreader generates can be seen as the number of coupons that could be collected up to seeing a duplicate; that is, a coupon that already is part of the collection. Moreover, the law presented in Proposition \ref{prop:distxxi} is the one of this quantity provided we start with a random number of coupon types. Note that for the case of a fixed number $n$ of coupons, Proposition \ref{prop:distxxi} claims that the mean number of coupons that could be collected up to seeing a duplicate, provided we already have one coupon, is given as $e^{n}\, n^{-(n-1)}\, \Gamma (n,n),$ where $\Gamma\left(k,x\right)$ is the incomplete gamma function. Therefore, some advances in the understanding of the (asymptotic) behavior of this distribution could bring some new results for rumor models like ours.



\section*{Acknowledgements}
This work has been developed with support of the Brazilian Federal Agency for Support and Evaluation of Graduate Education (CAPES), Financial Code 001. This work has been supported also by FAPESP (2017/10555-0). The authors thanks E. Lebensztayn for the fruitful discussions. 

\end{document}